\documentclass{amsart}
\usepackage{amsmath, amssymb, amsthm}
\usepackage{geometry}
\usepackage{mathrsfs}
\usepackage{xcolor}
\usepackage{booktabs}
\usepackage{hyperref}
\usepackage{tikz-cd}
\newcommand{\C}{\mathbb{C}}

\newcommand{\Tr}{\operatorname{Tr}}

\newcommand{\cH}{\mathcal{H}}

\newtheorem{lem}{Lemma}
\newtheorem{prop}{Proposition}
\newtheorem{thm}{Theorem}
\newtheorem{cor}{Corollary}
\newtheorem{remark}{Remark}
\newtheorem{dfn}{Definition}
\begin{document}

%\mainmatter             

 % start of a contribution
%
\title{The Infinite Sphere and Galois Belyi maps}  % abbreviated title (for running head)
\author{No\'emie C. Combe,} 

 \email{ noemie.combe@devinci.fr \\ De Vinci Higher Education \\ De Vinci Research Center, Paris, France}
%
%\authorrunning{No\'emie C.Combe} % abbreviated author list (for running head)
%\tocauthor{No\'emie C.Combe}
%

%\address{noemie.combe@devinci.fr, \\ De Vinci Higher Education, De Vinci Research Center, Paris, France},
%\texttt{https://sites.google.com/view/noemie-combe }

\maketitle              % typeset the title of the contribution

\begin{abstract}
We show that the space of Belyi maps admits a natural parametrization by an infinite-dimensional sphere arising from Voiculescu’s theory of noncommutative probability spaces. We show that this sphere decomposes into sectors, each of which corresponds to a class of Belyi maps distinguished up to isomorphism by their monodromy, encoded by a finite-index subgroup of $\mathbb{F}_2$. For Galois Belyi maps, our correspondence between spectral sectors of the infinite sphere and algebraic quotients of $\mathbb{F}_2$ yields a genuine bijection. Within this framework, distinct sectors of the sphere capture the algebraic constraints imposed on the monodromy, thereby providing a geometric organization of Belyi maps according to their associated group-theoretic data. 
\end{abstract}
\medskip 

{\bf keywords:} {Free probability, Belyi maps, Non Commutative Geometry}

\medskip
 
{\bf Acknowledgments:}{ I wish to thank Igor Nikolaev for drawing my attention on D. Voiculescu's free probability theory and to Vasily Dolgushev for many interesting discussions on Belyi maps and dessins d'enfant.}

\section{Introduction}
Belyi maps occupy a central position at the interface of algebraic geometry,
topology, and arithmetic. By Belyi’s theorem, algebraic curves defined over
$\overline{\mathbb{Q}}$ are precisely those admitting a branched covering of the
Riemann sphere ramified over at most three points, \cite{CMM,S}. 

From a topological perspective, such maps are classified up to isomorphism by their monodromy data. The latter is determined by a finite-index subgroup of $\mathbb{F}_2$, or equivalently, by a pair of permutations generating the monodromy group. In the Galois case, this subgroup is normal, and the map corresponds to a quotient $\mathbb{F}_2/N$. Despite their fundamental nature, the global geometry of the space of Belyi maps and the organization of its
monodromy classes remain poorly understood.

\, 

The aim of this paper is to introduce a geometric and probabilistic
construction that parametrizes the space of (topological and algebraic) Belyi maps in terms
of sectors of an infinite-dimensional sphere. More precisely, we construct an
infinite-dimensional sphere $\mathscr{S}_\infty$, equipped with a natural Haar-type
probability measure, obtained as a projective limit of finite-dimensional
matrix spheres defined in spaces of Hermitian matrices. Within this framework,
each sector of $\mathscr{S}_\infty$ is shown to correspond canonically to a quotient
algebra $\mathbb{F}_2/N$, where $N$ is a normal subgroup of the free group $\mathbb{F}_2$.
Words generated within a given sector are precisely the images of free words
under the quotient map $\mathbb{F}_2\to \mathbb{F}_2/N$. Since Belyi maps are classified by their
monodromy representations, this construction provides a probabilistic
mechanism for generating, in a unified manner, all classes of (Galois) Belyi maps with
prescribed monodromy.

\, 

The method of proof is to place this construction within a broader
geometric framework inspired by free probability theory \cite{V,DNV,T}. Free probability,
introduced by Voiculescu in the 1980s, was originally developed to address
fundamental problems in operator algebras, most notably the structure of free
group factors $L(\mathbb{F}_n)$. At its core, free probability furnishes a
noncommutative analogue of classical probability theory, in which random
variables are replaced by operators and independence by freeness. We show
that the probabilistic geometry of free probability provides a natural setting
in which the moduli of Belyi maps—and more generally Hurwitz maps—can be
organized and studied. In particular, the intrinsic noncommutative geometry
encoded in spaces of tracial states and matrix microstates reveals new
structural features of parameter spaces of branched covers that are invisible
from a purely combinatorial viewpoint.

\, 

Finally, let us mention that our work is motivated by, and conceptually related to, several
classical results at the intersection of random matrix theory, representation
theory, and enumerative geometry. Notably, the Okounkov integral \cite{O}, via the Miwa
correspondence, and the Itzykson--Zuber integral provide important precedents, see for instance \cite{GGN,N}.
The latter admits a combinatorial expansion whose structure was elucidated in
\cite{GGN}, where it is interpreted as a variant of the classical problem of
counting branched covers of the sphere with prescribed ramification data.
While these results highlight deep combinatorial connections between matrix
integrals and Hurwitz theory, the present work departs fundamentally from this
approach. Our construction is not combinatorial in nature; instead, it is
purely geometric, relying on the global structure of infinite-dimensional
spheres, operator-algebraic limits, and the geometry of noncommutative
probability spaces.

In this way, the paper proposes a new perspective on Belyi maps, replacing
discrete enumeration by a continuous geometric–probabilistic framework in
which monodromy, freeness, and large deviations play a central organizing role.

\paragraph{Main result.}
The central result of this paper establishes a precise correspondence between
large-deviation sectors of the infinite-dimensional sphere and algebraic
quotients of the free group.

\, 

Given a self-adjoint word $w\in\mathbb{C}\langle X,Y\rangle$ and a spectral law
$\nu\neq\mu_w$, where \(\mu_w\) is the empirical spectral distribution for the word \(w\) on large random matrices
from the spherical ensemble (see Sec. \ref{S:ESD}) and $\nu$ is a probability measure on the real numbers, which serves as the {\it target} spectral law for the empirical eigenvalue distribution of the matrix word $w(A_n, B_n)$. We show that conditioning the spherical ensemble on the event
that $w(A_n,B_n)$ has spectral distribution asymptotic to $\nu$ forces the
large-$n$ limit of the matrices to generate a quotient algebra $L(\mathbb{F}_2/N)$ for a
nontrivial normal subgroup $N\triangleleft \mathbb{F}_2$.

\, 

In this regime, the word $w$ converges not to the free operator $w(a,b)$, but to
its image in the quotient algebra, with spectral law exactly $\nu$.
Equivalently, each non-typical sector \ref{E:sector} of the infinite sphere gives rise to a
distinct algebraic quotient encoding the relations required to enforce the
prescribed spectral behavior.
As a consequence, the inverse-limit sector associated with $\nu$ generates the
free group factor $L(\mathbb{F}_2)$ if and only if $\nu=\mu_w$; otherwise it generates the
strictly smaller von Neumann algebra $L(\mathbb{F}_2/N)$ (see Theorem \ref{T:1}).

\, 

\paragraph{Significance for Belyi maps}
There is a bijection between isomorphism classes of Galois
Belyi maps and finite--index normal subgroups of the free group $\mathbb{F}_2$.
In the Galois case, our correspondence between spectral sectors of the infinite
sphere and algebraic quotients of $\mathbb{F}_2$ yields a genuine bijection: each sector
determines, and is determined by, a normal subgroup $N \triangleleft \mathbb{F}_2$, hence
by a unique Galois Belyi map up to isomorphism (see Theorem \ref{T:2}).

\, 

In contrast, for non--Galois Belyi maps the monodromy group alone does not
determine the covering, since distinct Belyi maps may have isomorphic abstract
monodromy groups.  In this setting, our construction provides a refinement of the
classical monodromy classification. Namely, spectral sectors distinguish Belyi
maps that share the same monodromy group but differ as covers, thereby yielding
a finer organization of non--Galois Belyi maps beyond group--theoretic
invariants.

\, 

\paragraph{Outline of the paper.} The aim is to prove the Theorems\ref{T:1} and \ref{T:2}. The tools of the proof are introduced in Sect~\ref{sec:preliminaries}-- ~\ref{sec:spherical}.
Precisely, in Section~\ref{sec:preliminaries} we review the necessary background on Belyi maps,
and free probability theory, including the construction of matrix
microstates and the basic properties of the free group factors $L(\mathbb{F}_n)$. In
Section~\ref{sec:spherical} we introduce the finite- and infinite-dimensional
spherical ensembles and define the microstate sectors $\Gamma_n(\nu,\varepsilon)$
that parametrize classes of Belyi maps according to their monodromy. 
Section~\ref{sec:main} contains the proof of the main theorem, establishing the correspondence
between large-deviation sectors and quotient algebras $L(\mathbb{F}_2/N)$, and
illustrates the resulting algebraic constraints on words in the free group. 
Finally, we provide a discussion of implications for the geometry of moduli spaces of (algebraic) Belyi maps, based on our construction.

\,

\section{Belyi Maps and Free Probability Preliminaries}\label{sec:preliminaries}
\subsection{Motivation}

Belyi maps play a central role in algebraic geometry by providing a bridge
between complex geometry, topology, and arithmetic. By Belyi’s theorem, a
smooth complex curve is defined over $\overline{\mathbb{Q}}$ if and only if it
admits a branched covering of the Riemann sphere ramified over at most three
points. From a topological viewpoint, such maps are classified by their
monodromy representations, which may be encoded by words in the free group
$\mathbb{F}_2$, or more generally by their images in quotients $\mathbb{F}_2/N$ for normal
subgroups $N\triangleleft \mathbb{F}_2$. This monodromy description reduces the study of
Belyi maps to discrete group-theoretic data, but offers limited insight into
the global organization of the space of all such maps.

\, 

The complexity and abundance of Belyi maps suggest the need for a framework
that goes beyond combinatorial classification and captures their global
structure in a geometric manner. In particular, one would like a mechanism
that organizes Belyi maps according to their monodromy while allowing for
continuous deformations, asymptotic limits, and natural measures. Such a
framework should reflect both the rigidity imposed by algebraic constraints
and the flexibility inherent in large families of coverings.

\, 

The conceptual goal of this paper is to provide such a framework by connecting
the monodromy classification of Belyi maps with the notion of {\it microstate} constructions arising
in random matrix theory and free probability. By interpreting words in the free
group as limits of matrix words evaluated on spherical ensembles, we relate
classes of Belyi maps to geometric sectors of an infinite-dimensional sphere.
In this way, probabilistic structures on matrix ensembles give rise to a
geometric organization of Belyi maps indexed by quotients of the free group,
thereby linking monodromy data with noncommutative probability spaces.

\subsection{Belyi maps, monodromy, and Hurwitz spaces}
\subsubsection{Belyi maps: topological and algebraic viewpoints}
Belyi maps admit equivalent formulations from both topological and algebraic perspectives.

\, 
\begin{itemize}
\item From the topological viewpoint a \emph{Belyi map} is a finite ramified covering
\[
f \colon X \longrightarrow \mathbb{C}\mathbb{P}^1
\]
of compact Riemann surfaces, branched only over three points.
By composing with a M\"obius transformation, the branch locus may be taken to be
$\{0,1,\infty\}$. Equivalently, a Belyi map is a three--point branched cover of the Riemann sphere. We can consider Belyi maps as a subclass of Hurwitz maps.

\item Algebraically, a \emph{Belyi map} is a nonconstant morphism $f \colon X \longrightarrow \mathbb{P}^1$
of smooth, projective algebraic curves defined over $\overline{\mathbb{Q}}$, \'{e}tale outside the set $\{0,1,\infty\}$.
\end{itemize}
Belyi’s theorem asserts that a compact Riemann surface admits such a morphism
if and only if it can be defined over $\overline{\mathbb{Q}}$.

However, note that when the base field is $\mathbb{C}$, these definitions are equivalent. Indeed, 
every three--point topological cover admits an algebraic structure defined over
$\overline{\mathbb{Q}}$, and conversely every algebraic Belyi map determines a
topological covering of $\mathbb{C}\mathbb{P}^1$ branched at three points.

\subsubsection{Monodromy classification and the free group $\mathbb{F}_2$}

Let ${\bf P}_\ast = \mathbb{C}\mathbb{P}^1 \setminus \{0,1,\infty\}.$ The fundamental group of ${\bf P}_\ast$  is the free group on two generators,
\[
\pi_1({\bf P}_\ast) \cong \mathbb{F}_2,
\]
generated by homotopy classes of small loops around
$0$ and $1$, respectively denoted $\gamma_0$ and $\gamma_1$.

\, 

Given a degree--$d$ Belyi map $f$, the associated monodromy representation is
\[
\rho_f \colon \mathbb{F}_2 \longrightarrow \mathbb{S}_d,
\]
where $\mathbb{S}_d$ is the  permutation group with $d$ elements. 
Up to isomorphism, the map $f$ is determined by the pair of permutations
\[
s_0 = \rho_f(\gamma_0), \qquad
s_1 = \rho_f(\gamma_1),
\]
and $s_\infty := (s_0 s_1)^{-1}$, which encodes the monodromy around
the point at infinity. Note that any transitive permutation triple
$(s_0,s_1,s_\infty) \in \mathbb{S}_d^3$ satisfying
$s_0 s_1 s_\infty = 1$ arises from a Belyi map.

%Let $N = \ker \rho_f \;\triangleleft\; \mathbb{F}_2.$ Then $N$ is a finite--index normal subgroup, and
%\[
%\mathbb{F}_2 / N \;\cong\; \mathrm{Deck}(f) \;\subset\; \mathbb{S}_d.
%\]
Degree $d$ Belyi maps can be described in three equivalent ways by:
\begin{itemize}
\item finite--index subgroups $\tilde{J} \leq \mathbb{F}_2$.
\item Finite quotients $\mathbb{F}_2 \twoheadrightarrow G$ together with a generating
pair $(g_0,g_1)$, where  $G\cong \mathbb{F}_2/N$ with  $N=ker(\mathbb{F}_2\to G)$ stands for the monodromy group of the Belyi map. 
\item Permutation triples $(\sigma_0,\sigma_1,\sigma_\infty)$ in the permutation group $\mathbb{S}_d$
satisfying $\sigma_0\sigma_1\sigma_\infty=1$, modulo simultaneous conjugation.
\end{itemize}

\subsubsection{Galois Belyi maps and normal subgroups}

A Belyi map $f \colon X \longrightarrow \mathbb{P}^1$
is called \emph{Galois} if the monodromy representation
\[
\rho \colon \pi_1({\bf P}_\ast) \longrightarrow \mathrm{Deck}(f)
\]
is surjective with kernel
\[
N = \ker \rho
\]
a normal subgroup of finite index in $\mathbb{F}_2$.

In this case, the deck transformation group is canonically isomorphic to
\[
\mathrm{Deck}(f) \cong \mathbb{F}_2 / N.
\]

Similarly, given any normal subgroup $N \triangleleft \mathbb{F}_2$ of finite index,
one may construct the corresponding Galois covering of ${\bf P}_\ast$,
which compactifies uniquely to a Galois Belyi map branched only at
$\{0,1,\infty\}$. Hence there is a bijection between isomorphism classes of Galois Belyi maps and
finite--index normal subgroups of the free group $\mathbb{F}_2$. We refer to \cite{S} for a full exposition on the topic. 

\subsubsection{Hurwitz spaces}

Hurwitz spaces provide moduli spaces parametrizing branched coverings of the
Riemann sphere with prescribed branching data and our construction van be easily generalized to the space of Hurwitz maps.
 In the case of Belyi maps, these
spaces decompose according to degree, monodromy group, and ramification type.
From a combinatorial viewpoint, points of Hurwitz spaces correspond to
equivalence classes of permutation data or dessins d’enfants, while from a
geometric perspective they encode families of algebraic curves together with
distinguished maps to $\mathbb{P}^1$.

Although Hurwitz spaces are traditionally studied via combinatorial or
algebro--geometric methods, their structure also suggests the existence of a
global geometric organization indexed by monodromy data. This observation
motivates the probabilistic and noncommutative approach developed in the
present work.

\subsection{Finite-Dimensional Matrix Ensembles}

\subsubsection{On the Measurable Space $(\cH_m,d{\boldsymbol{\mu}})$.}
The basic object of our construction is the real vector space $\cH_m$ of
$m\times m$ Hermitian matrices, equipped with the Hilbert--Schmidt inner product
$\langle M_1,M_2\rangle=\Tr(M_1M_2)$. Consider the space of probability density functions on \(\cH_m\):
\[
\mathcal{P}(\cH_m)
=
\left\{
\rho : \cH_m \to \mathbb{R}_{+}
\;\middle|\;
\int_{\cH_m} \rho(M)\, dM = 1,\;
\rho \in L^1(\cH_m)
\right\},
\]
where \(\rho\) is a density corresponding to a probability measure $d{\boldsymbol{\mu}}^{\rho}_m(M) = \rho(M)\, dM$, where $M\in \cH_m$.

There exists a family of probability density functions of exponential type on $\mathcal{H}_m$, defined as follows:
\begin{equation}\label{E:1}
    d{\boldsymbol{\mu}}_m(M)
= \frac{1}{Z_m} 
\exp\!\Bigl( - \frac{m}{2} \operatorname{Tr}(M^2) \Bigr)\, dM,
\end{equation}
where $dM$ is the Lebesgue measure on the independent real parameters of $M$, and 
$Z_m$ is the normalization constant. We call any element in the measurable space a Wigner matrix. 

Furthermore, the space of those measures $(\mu_m,\tilde{d})$ forms a metric space, where the metric is defined by: 
\[\tilde{d}({\boldsymbol{\mu}}_m,{\boldsymbol{\mu}}'_m)=2\arccos\int_{\mathcal{H}_m}\sqrt{d{\boldsymbol{\mu}}_m(dM)}\sqrt{d{\boldsymbol{\mu}}'_m(dM)}.\]

\subsubsection{Wigner Matrices}
We focus specifically on Wigner random matrices. These Wigner random matrices form a model where entries are complex random variables and the eigenvalues follow semicircular law \cite{AG,G,T}. It is defined as Hermitian square matrix. The structure of such a matrix is as follows:

\begin{enumerate} 
\item the entries on and above the diagonal are independent random variables, 
\item these random variables have zero mean and a finite, normalized variance, 
\item the diagonal entries may follow a different distribution than the off-diagonal ones, but they also have zero mean.
\end{enumerate}

\subsection{Finite-dimensional matrix ensembles and the spherical ensemble}

We consider the associated spherical ensemble
$\mathscr{S}_m\subset \cH_m$, defined by
\[
\mathscr{S}_m
=
\Bigl\{
M\in \cH_m \;:\; \frac{1}{m}\Tr(M^2)=1
\Bigr\}.
\]
Thus $\mathscr{S}_m$ is the unit sphere in the $m^2$--dimensional Euclidean space $\cH_m$.
It carries a canonical $\mathrm{U}(m)$--invariant probability measure $\sigma_m$,
which we refer to as the uniform (or Haar) measure on $\mathscr{S}_m$.

The normalization $\frac{1}{m}\Tr(M^2)=1$ ensures that the spectrum of a typical
matrix in $\mathscr{S}_m$ remains of order one as $m\to\infty$, leading to a nontrivial
large--$m$ limit. Moreover, the uniform measure $\sigma_m$ is characterized by
maximal symmetry: it is the unique probability measure on $\mathscr{S}_m$ invariant under
the adjoint action of $\mathrm{U}(m)$, and may be viewed as the microcanonical
ensemble associated with the quadratic energy constraint $\Tr(M^2)=m$.

From the probabilistic point of view, the spherical ensemble is asymptotically
equivalent to the Eq.\eqref{E:1} (also known as Gaussian Unitary Ensemble). More precisely, for large $m$,
sampling a matrix uniformly from $\mathscr{S}_m$ is equivalent, up to negligible
fluctuations, to sampling from the Gaussian Unitary Ensemble and conditioning on the Hilbert--Schmidt
norm. 

\subsubsection{Semicircular distribution}
Our interest in these Wigner matrices takes root in a statement by D.Voiculescu \cite{V,DNV}, connecting random matrix theory with non commutative geometry (see section \ref{S:Voiculescu}). Take a pair of independent Wigner $m\times m$ matrices $(A_m,B_m)$. Then, as $m\to \infty$, the joint non-commutative distribution of $(A_m,B_m)$ converges to that of a pair $(a,b)$ of freely independent non-commutative random variables in a suitable $C^*-$ or $W^*$-probability space.  In particular, an interesting approach is to use a corollary of the classical Wigner Semicircle Law. This states that the empirical spectral distribution of a normalized Wigner matrix converges to the Wigner semicircular distribution:
\[d{\boldsymbol\sigma}(t)=\frac{1}{2\pi}\sqrt{4-t^2}dt, \quad \text{for}\quad |t|\leq 2.\]
This motivates the fact that we will be working in priority with infinite dimensional spheres. 

\subsection{Empirical spectral distributions}\label{S:ESD}

The empirical spectral distribution provides the fundamental link between
linear algebra and probability theory in random matrix models.
\begin{dfn}
 Given an
$m\times m$ Hermitian matrix $M$ with eigenvalues
$\lambda_1(M),\dots,\lambda_m(M)$, its empirical spectral distribution (ESD) is
defined by
\[
\mu_M
=
\frac{1}{m}\sum_{i=1}^m \delta_{\lambda_i(M)}.
\]
\end{dfn}
Thus $\mu_M$ records the normalized eigenvalue counting measure of $M$.

\, 

By a classical result of Wigner, if $M$ is sampled from the spherical
ensemble $\mathscr{S}_m$, then $\mu_M$ converges almost surely, as $m\to\infty$, to the
semicircular distribution of variance one. This convergence underlies the
appearance of semicircular elements as canonical limits of large random
Hermitian matrices.

\, 

More generally, given a pair $(A_m,B_m)\in \mathscr{S}_m\times \mathscr{S}_m$ and a self-adjoint word
$w\in\mathbb{C}\langle X,Y\rangle$, the matrix $w(A_m,B_m)$ is Hermitian and
admits an empirical spectral distribution
\[
\mu_{w(A_m,B_m)}
=
\frac{1}{m}\sum_{i=1}^m
\delta_{\lambda_i\bigl(w(A_m,B_m)\bigr)}.
\]
This measure is the primary spectral observable associated with the word $w$.
Its large--$m$ limit, when it exists, defines a probability measure $\nu$ on
$\mathbb{R}$ which encodes the spectral law of the corresponding operator
$w(u,v)$ in the noncommutative limit. As will be seen below, the measure $\nu$
serves as a spectral signature of the algebraic relations satisfied by the
limiting operators and plays a central role in the decomposition of the
infinite--dimensional sphere into microstate sectors.

\subsection{Microstate sectors and noncommutative distributions}

The conceptual framework underlying our construction is Voiculescu’s
microstate philosophy. In this approach, abstract noncommutative probability
spaces are approximated by concrete families of finite-dimensional matrices.
A microstate of size $m$ is given by a pair
\[
(A_m,B_m)\in \mathscr{S}_m\times \mathscr{S}_m,
\]
where $\mathscr{S}_m$ denotes the spherical ensemble of $m\times m$ Hermitian matrices.

Such a matrix pair encodes a noncommutative probability distribution through
its moments. More precisely, for any noncommutative polynomial
$P\in\mathbb{C}\langle X,Y\rangle$, the quantity
\[
\frac{1}{m}\Tr\bigl(P(A_m,B_m)\bigr)
\]
defines a joint moment of the pair $(A_m,B_m)$. Convergence of all such moments
as $m\to\infty$ determines a limiting $*$--distribution in the sense of
noncommutative probability theory.

\begin{dfn}
Fixing a self-adjoint word $w\in\mathbb{C}\langle X,Y\rangle$, a probability
measure $\nu$ on $\mathbb{R}$ and \(\varepsilon>0\), we define the associated microstate sector of size $n$ by
\begin{equation}\label{E:sector}
\Gamma_m(\nu,\varepsilon)
=
\Bigl\{
(A_m,B_m)\in \mathscr{S}_m\times \mathscr{S}_m \;:\;
d\bigl(\mu_w(A_m,B_m),\nu\bigr)<\varepsilon
\Bigr\},
\end{equation}
where $\mu_w(A_m,B_m)$ denotes the empirical spectral measure of
$w(A_m,B_m)$ and $d$ is any metric compatible with the weak topology on
probability measures (for instance, the Lévy--Prokhorov distance). Thus $\Gamma_m(\nu,\varepsilon)$ consists of those
microstates whose spectral behavior for the word $w$ approximates the target
law $\nu$.
\end{dfn}

Thus \(\Gamma_m(\nu,\varepsilon)\) consists of those matrix pairs whose spectral
behavior under the word \(w\) is constrained to lie within an
\(\varepsilon\)-neighborhood of the prescribed law \(\nu\).
Geometrically, these sets form thin sectors inside the product sphere
\(\mathscr{S}_m\times \mathscr{S}_m\).

In the absence of constraints, that is, when $(A_m,B_m)$ is sampled from the
full product sphere $\mathscr{S}_m\times \mathscr{S}_m$ with its uniform measure, the joint
$*$--distribution converges almost surely, as $m\to\infty$, to that of a pair
$(a,b)$ of freely independent semicircular elements of variance one. The
limiting operators generate the free group factor $L(\mathbb{F}_2)$, and the associated
noncommutative law represents the typical, or equilibrium, distribution in the
microstate picture.

\subsection{From microstates to monodromy groups}
In this subsection, we etablish the correspondence between random matrix microstates and the algebraic
classification of Belyi maps proceeds in three conceptual steps, which will be used later in the proof of our main statement.

\paragraph{Step 1 (Matrix data $\longleftrightarrow$ Algebra).}
Consider a microstate sector $\Gamma_m(\nu,\varepsilon)$ associated with a
self-adjoint word $w\in\mathbb{C}\langle X,Y\rangle$ and a target spectral
law $\nu$. In the large--$m$ limit, typical sequences $(A_m,B_m)\in
\Gamma_m(\nu,\varepsilon)$ converge in $*$--distribution to a pair of operators
\[
(u,v) \in L(\mathbb{F}_2/N_\nu),
\]
where $L(\mathbb{F}_2/N_\nu)$ is the group von Neumann algebra of a quotient of the free
group $\mathbb{F}_2$ by a normal subgroup $N_\nu$. The spectral law $\nu$ of the
operator $w(u,v)$ is preserved under this limit and serves as the primary
invariant of the sector.

\paragraph{Step 2 (Algebra $\longleftrightarrow$ Group theory).}
The quotient $\mathbb{F}_2/N_\nu$ is a finite group when $\nu$ arises from a
microstate sector corresponding to a Galois Belyi map. In this context, the
quotient coincides precisely with the monodromy group of the covering.

\paragraph{Step 3 (Encoding monodromy via spectra).}
Fixing a spectral law $\nu$ selects a unique normal
subgroup $N_\nu \triangleleft \mathbb{F}_2$, which in turn determines a unique Galois
Belyi map. Its monodromy group is then canonically identified with the quotient
\[
\mathbb{F}_2/N_\nu.
\]
Thus, in this framework, spectral data extracted from matrix words
encodes the full monodromy information of the associated Belyi maps.

\subsection{Geometric Properties}

\begin{lem}
The probability density measure defined in Eq. \ref{E:1} is invariant under unitary conjugation. 
\end{lem}
\begin{proof}
The probability density contains the term
\[
\exp\!\left( -\frac{n}{2}\, \mathrm{Tr}(M^{2}) \right).
\]
Under a unitary conjugation \(M' = U M U^\dagger\), we have
\[
\mathrm{Tr}\!\left( (M')^{2} \right)
= 
\mathrm{Tr}\!\left( (U M U^\dagger)(U M U^\dagger) \right)
=
\mathrm{Tr}\!\left( U M^{2} U^\dagger \right)
=
\mathrm{Tr}(M^{2}),
\]
since the trace is cyclic.

Thus, the exponential weight is invariant:
\[
\exp\!\left( -\frac{n}{2}\, \mathrm{Tr}\!\left( (M')^{2} \right) \right)
=
\exp\!\left( -\frac{n}{2}\, \mathrm{Tr}(M^{2}) \right).
\]

Conjugation by a unitary matrix \(U\) defines a linear map
\[
T_{U} : \cH_{n} \to \cH_{n}, 
\qquad 
T_{U}(M) = U M U^{\dagger},
\]
on the real vector space \(\cH_{n}\) of Hermitian matrices.  
This space carries the natural Frobenius (Hilbert--Schmidt) inner product
\[
\langle A , B \rangle = \mathrm{Tr}(AB),
\]
which is real-valued for Hermitian matrices.

Under the transformation \(T_{U}\), one has
\[
\langle T_{U}(A), T_{U}(B) \rangle
=
\mathrm{Tr}\!\left( U A U^{\dagger} \cdot U B U^{\dagger} \right)
=
\mathrm{Tr}\!\left( U A B U^{\dagger} \right)
=
\mathrm{Tr}(AB)
=
\langle A, B \rangle ,
\]
where we used cyclicity of the trace.  
Thus \(T_{U}\) is an orthogonal transformation of the real inner-product space
\(\cH_{n} \cong \mathbb{R}^{n^{2}}\).

Any orthogonal linear transformation has determinant \(\pm 1\), and therefore
preserves Lebesgue measure.  If \(J\) denotes the Jacobian matrix of \(T_{U}\),
then
\[
dM' = |\det J|\, dM = dM.
\]

Hence unitary conjugation preserves the Lebesgue measure on \(\cH_{n}\).

Moreover, the Lebesgue measure \(dM\) on the space of Hermitian matrices
is invariant under the transformation \(M \mapsto U M U^\dagger\).  
Hence,
\[
d\Gamma_{n}(UMU^\dagger)
=
\frac{1}{Z_{n}}
\exp\!\left( -\frac{n}{2}\,\mathrm{Tr}(M^{2}) \right)
\, d(UMU^\dagger)
=
\frac{1}{Z_{n}}
\exp\!\left( -\frac{n}{2}\,\mathrm{Tr}(M^{2}) \right)
\, dM
=
d\Gamma_{n}(M).
\]

Thus, the measure is unitarily invariant:
\[d\Gamma_{n}(UMU^\dagger) = d\Gamma_{n}(M)
\qquad \text{for all unitary } U.\]

Therefore,
\[
\exp\!\left( -\frac{n}{2}\,\mathrm{Tr}\!\left( (M')^{2} \right) \right)
=
\exp\!\left( -\frac{n}{2}\,\mathrm{Tr}(M^{2}) \right).
\]

\end{proof}

\subsection{Free Probability Theory} 
\subsubsection{Noncommutative probability space}
In the algebraic framework of free probability, a \emph{noncommutative probability space} is a pair $(\mathcal{A},\varphi)$, where $\mathcal{A}$ is a unital $*$-algebra over $\mathbb{C}$ and $\varphi:\mathcal{A}\to \mathbb{C}$ is a positive linear functional satisfying $\varphi(1)=1$. The elements of $\mathcal{A}$ are called \emph{noncommutative random variables}.

\, 

\subsubsection{Free independence}
Let $\{\mathcal{A}_i\}_{i\in I}$ be unital subalgebras of $\mathcal{A}$, i.e., $1_\mathcal{A}\in \mathcal{A}_i$ for all $i$. These subalgebras are said to be \emph{free} (or \emph{freely independent}) if, for any $k\ge 1$ and elements $a_j\in \mathcal{A}_{i(j)}$ satisfying $\varphi(a_j)=0$, the mixed moment
\[
\varphi(a_1 \cdots a_k) = 0
\]
whenever consecutive indices differ: $i(1)\neq i(2)\neq \dots \neq i(k)$. Freeness is the noncommutative analogue of classical independence.

\, 
\subsubsection{Semicircularity}
An element $x \in (\mathcal{A},\varphi)$ is called \emph{semicircular} with variance $t>0$ if it is self-adjoint and its moments satisfy
\[
\varphi(x^k) = \int_{-2\sqrt{t}}^{2\sqrt{t}} s^k \, d{\boldsymbol\sigma}_t(s),
\]
where
\[
d{\boldsymbol\sigma}_t(s) = \frac{1}{2\pi t} \sqrt{4t - s^2} \, ds
\]
denotes the semicircular distribution. Semicircular elements play the role of Gaussian variables in free probability.

\subsubsection{Free groups}
Let $\mathbb{F}_n$ denote the free group on $n$ generators, and let $\lambda:\mathbb{F}_n \to B(\ell^2(\mathbb{F}_n))$ denote the left regular representation. The von Neumann algebra generated by $\lambda(\mathbb{F}_n)$, denoted $L(\mathbb{F}_n)$, is a type $\mathrm{II}_1$ factor equipped with its canonical trace $\tau$ defined by $\tau(g) = \delta_{g,e}$. The group generators become unitaries in $L(\mathbb{F}_n)$, and their self-adjoint parts, e.g., $(a+a^{-1})/2$, are semicircular and freely independent.

\begin{remark}[Limiting Behavior and $L(\mathbb{F}_2)$]
The factor $L(\mathbb{F}_2)$ arises naturally as the limiting object for two independent large random Hermitian matrices. By Voiculescu's asymptotic freeness theorem, if $A_m$ and $B_m$ are independent Hermitian matrices from, e.g., the spherical ensemble $\mathscr{S}_m$, then the pair $(A_m,B_m)$ converges in distribution, as $m\to\infty$, to a pair $(a,b)$ of freely independent semicircular elements. The von Neumann algebra generated by $a$ and $b$ is isomorphic to $L(\mathbb{F}_2)$, and any noncommutative polynomial $w(a,b)$ lies naturally in $L(\mathbb{F}_2)$, with its spectral distribution determined by the trace $\tau$.
\end{remark}

\subsection{How each sector corresponds to a quotient}

Let $\nu$ be a spectral law arising from a microstate sector $\Gamma(\nu,\varepsilon)$. Our main theorem identifies a unique nontrivial normal subgroup 
$N_\nu \triangleleft \mathbb{F}_2$,
 
and the limiting operators $(u,v)$ lie in the group von Neumann algebra $L(\mathbb{F}_2/N_\nu)$.

The crucial map is the canonical quotient
\[
\pi_\nu: \mathbb{F}_2 \to \mathbb{F}_2/N_\nu,
\] 
which extends linearly to the group algebra and further to the von Neumann algebra:
\[
\pi_\nu : \mathbb{C}[\mathbb{F}_2] \to \mathbb{C}[\mathbb{F}_2/N_\nu], \qquad 
\pi_\nu : L(\mathbb{F}_2) \to L(\mathbb{F}_2/N_\nu).
\]

Any word $w \in \mathbb{C}\langle X,Y\rangle$ is first interpreted in $L(\mathbb{F}_2)$ via the free semicircular system $(a,b)$:
\[
w \mapsto w(a,b).
\]
However, in the sector corresponding to $\nu$, the limit of the matrix word is not $w(a,b)$ but rather its image under the quotient:
\[
\lim_{m\to\infty} w(A_m,B_m) = w(u,v) = \pi_\nu(w(a,b)) \in L(\mathbb{F}_2/N_\nu).
\]
The spectral measure of $w(u,v)$ is exactly $\nu$, which generally differs from the free law $\mu_w$ because the relations in $N_\nu$ alter the operator's spectrum.

\subsubsection{Interpretation}
Each sector $\Gamma(\nu,\varepsilon)$ on the infinite sphere $\mathscr{S}_\infty$ is a geometric locus for all Belyi maps whose Galois closure has monodromy group $\mathbb{F}_2/N_\nu$. The sector organizes covers by their shared maximal Galois piece.
    
Within a sector, each word evaluates via the quotient map $\pi_\nu$. The relations in $N_\nu$ (e.g., $r(a,b)=I$) constrain the operator algebra, yielding the specific spectral law $\nu$.
    
Concerning, probabilistic (large deviations): the probability of landing in sector $\Gamma_m(\nu,\varepsilon)$ satisfies
    \[
    \mathbb{P}[\Gamma_m(\nu,\varepsilon)] \sim \exp(-m^2 I(\nu)),
    \]
    where the rate function
    \[
    I(\nu) = \chi(a,b) - \chi(\nu)
    \]
    measures the free-entropy deficit relative to the free law.

We summarize this in the following table. 

\begin{center}
\begin{tabular}{|c|c|c|}
\hline
\textbf{Spectral} & \textbf{Algebraic Object} & \textbf{Belyi Map} \\
\textbf{/Geometric Object} &&\textbf{Interpretation}\\
\hline
Full sphere  & Free group factor & Moduli space of all covers \\
$\mathscr{S}_\infty$&  $L(\mathbb{F}_2)$ & \\
Sector $\Gamma(\nu,\varepsilon)$ & Quotient algebra $L(\mathbb{F}_2/N_\nu)$ & Family of covers with Galois closure $\mathbb{F}_2/N_\nu$ \\
Spectral law $\nu$ & Kernel $N_\nu$ of $\pi_\nu$ & Invariant encoding \\
&&defining relations of Galois closure \\

\hline
\end{tabular}
\end{center}

\section{Spherical Ensembles and Microstate Sectors}
\label{sec:spherical}
\subsection{Finite-Dimensional Spherical Ensembles}

Let $\cH_m$ denote the real vector space of $m\times m$ complex Hermitian matrices. Equip $\cH_m$ with the Hilbert--Schmidt inner product and norm:
\[
\langle A, B \rangle_{\mathrm{HS}} = \mathrm{Tr}(AB), 
\qquad
\|A\|_{\mathrm{HS}}^2 = \mathrm{Tr}(A^2).
\]
The \emph{matrix sphere} is
\[
\mathscr{S}_m := \Big\{ A \in \cH_m : \frac{1}{m} \mathrm{Tr}(A^2) = 1 \Big\}.
\]
Equivalently, $\mathscr{S}_m$ is the Hilbert--Schmidt sphere of radius $\sqrt{m}$ in $\cH_m$, a smooth compact manifold of real dimension $m^2 - 1$.

Let $(\mathscr{S}_m,\sigma_m)$ be the sphere equipped with its normalized Haar (uniform) measure. 
Then, $\sigma_m$ coincides with the conditional distribution of a GUE matrix given the event
\[
\frac{1}{m}\mathrm{Tr}(G_m^2) = 1.
\]

\begin{remark}[Unitary Invariance and Uniform Measure]
The unitary group $U(m)$ acts on $\mathscr{S}_m$ by conjugation:
\[
A \mapsto U A U^*, \quad U \in U(m),
\]
preserving the sphere. The unique Borel probability measure on $\mathscr{S}_m$ invariant under this action is denoted $\sigma_m$ (the \emph{normalized Haar measure} or uniform distribution). It can also be obtained as:
\begin{enumerate}
    \item the normalized surface (area) measure induced by $H_m$; 
    \item the conditional law of a GUE matrix restricted to $\mathscr{S}_m$.
\end{enumerate}
\end{remark}

Let $A_m,B_m \in \mathscr{S}_m$ be independent random matrices sampled from $\sigma_m$. Then:
\begin{enumerate}
    \item \emph{Individual Spectral Law:} The empirical spectral distribution
    \[
    \mu_{A_m} := \frac{1}{m} \sum_{i=1}^m \delta_{\lambda_i(A_m)}
    \]
    converges weakly almost surely to the Wigner semicircular law
    \[
    d{\boldsymbol\sigma}(t) = \frac{1}{2\pi} \sqrt{4-t^2}\, dt, \qquad t \in [-2,2].
    \]
    
    \item \emph{Joint Unitary Invariance:} For any $U \in U(m)$,
    \[
    (U A_m U^*, U B_m U^*) \stackrel{d}{=} (A_m, B_m).
    \]
    
    \item \emph{Asymptotic Freeness:} As $m \to \infty$, $(A_m,B_m)$ converges in joint $*$-distribution to a pair $(a,b)$ of freely independent semicircular operators.
\end{enumerate}

\subsection{Noncommutative Distributions and Words}\label{S:Voiculescu}

Let $A_n,B_n\in \cH_n$ be random matrices. The normalized trace of a matrix $M \in \cH_n$ is
\[
\tau_n(M) := \frac{1}{n} \mathrm{Tr}(M).
\]

The \emph{joint *-distribution} of a pair $(A_n, B_n)$ is the collection of all normalized *-moments
\[
\tau_n\big(M_{i_1}^{\epsilon_1} M_{i_2}^{\epsilon_2} \cdots M_{i_k}^{\epsilon_k}\big),
\]
where $M_{i_j} \in \{A_n, B_n\}$ and $\epsilon_j \in \{1, *\}$ (with $M^1 = M$ and $M^* = $ conjugate transpose). This data completely characterizes the non-commutative probability distribution with respect to $\tau_n$.

For a non-commutative polynomial (word) $w \in \mathbb{C}\langle X,Y \rangle$, the evaluation $w(A_n, B_n)$ is obtained by substituting $X \mapsto A_n$, $Y \mapsto B_n$, respecting the non-commutative product order. If $w$ is self-adjoint (i.e., $w = w^*$ under $X^*=X, Y^*=Y$), then $w(A_n, B_n)$ is Hermitian.

For a self-adjoint word $w$, the \emph{empirical spectral measure} of $w(A_n, B_n)$ is
\[
\mu_{w(A_n,B_n)} := \frac{1}{n} \sum_{i=1}^n \delta_{\lambda_i},
\]
where $\lambda_1,\dots,\lambda_n$ are the eigenvalues of $w(A_n, B_n)$. This measure describes the eigenvalue distribution of the matrix word.

By Voiculescu, we have that: 
\begin{prop}let $(A_n, B_n)$ be a sequence of independent random matrix pairs drawn from a unitarily invariant ensemble such as the spherical ensemble $\mathscr{S}_n \times \mathscr{S}_n$. Then, almost surely as $n \to \infty$:
\begin{enumerate}
    \item \emph{Joint Convergence:} $(A_n, B_n)$ converges in joint *-distribution to a pair $(a,b)$ of freely independent semicircular operators of variance 1.
    \item \emph{Spectral Convergence:} For any self-adjoint word $w \in \mathbb{C}\langle X,Y\rangle$, the empirical spectral measure $\mu_{w(A_n,B_n)}$ converges weakly to a deterministic probability measure $\mu_w$ on $\mathbb{R}$, which is the spectral measure of $w(a,b)$ in the von Neumann algebra generated by $(a,b)$ (isomorphic to $L(\mathbb{F}_2)$).
    \item \emph{Moment Convergence:} For all $k \in \mathbb{N}$,
    \[
    \tau_n\big(w(A_n,B_n)^k\big) \longrightarrow \tau\big(w(a,b)^k\big) = \int_\mathbb{R} t^k \, d\mu_w(t),
    \]
    where $\tau$ is the canonical trace on $L(\mathbb{F}_2)$.
\end{enumerate}

\end{prop}

\begin{proof}
    This follows from the construction given in \cite{V} and \cite{AG}.
\end{proof}

The measure $\mu_w$ is called the \emph{free law} of the word $w$ in two free semicirculars. It is the non-commutative analogue of the distribution of a polynomial in two independent Gaussian random variables.
\medskip

\subsection{The Infinite-Dimensional Sphere} 

As $n$ tends to infinity, the dimension of the sphere $\sc\mathscr{S}_n$ grows and converges in a suitable sense to an infinite-dimensional sphere $\sc\mathscr{S}_{\infty}$ in the Hilbert space of operators. The joint non-commutative distribution of $(A_n,B_n)$ converges to that of the free semicircular pair $(a,b)$. The algebra  \(L(\mathbb{F}_2)\) is viewed therefore as the non-commutative function algebra on an infinite-dimensional  sphere.

\medskip

\subsubsection{Construction as a Projective Limit}
For each $n \in \mathbb{N}$, let $(\mathscr{S}_n, \sigma_n)$ denote the $n$-dimensional matrix sphere equipped with the uniform Haar measure. For $m > n$, define the embedding
\[
\iota_{m,n} : \mathscr{S}_n \hookrightarrow \mathscr{S}_m, \qquad 
\iota_{m,n}(A) := \sqrt{\frac{n}{m}} 
\begin{pmatrix}
A & 0 \\[2mm]
0 & 0
\end{pmatrix}.
\]
This map embeds a smaller matrix into the upper-left corner of a larger one and rescales it so that the Hilbert--Schmidt norm condition $(1/m)\mathrm{Tr}(\cdot^2)=1$ is satisfied. The embeddings are compatible in the sense that
\[
\iota_{\ell,m} \circ \iota_{m,n} = \iota_{\ell,n}, \quad \text{for all } n < m < \ell.
\]

\begin{dfn}[Infinite-Dimensional Sphere]
The \emph{infinite-dimensional sphere} is defined as the projective limit
\[
\mathscr{S}_\infty := \varprojlim (\mathscr{S}_n, \iota_{m,n}) = 
\Big\{ (A_1, A_2, \dots) \in \prod_{n=1}^\infty \mathscr{S}_n \;\Big|\; \iota_{m,n}(A_n) = A_m \text{ for all } m>n \Big\}.
\]
\end{dfn}

A point $A_\infty \in \mathscr{S}_\infty$ is a coherent sequence of matrices of increasing size, where each finite section is determined by its predecessors.

The uniform measures $\sigma_n$ on the finite-dimensional spheres induce a natural measure $\sigma_\infty$ on $\mathscr{S}_\infty$:

\begin{remark}
The embeddings $\iota_{m,n}$ are measure-preserving. Indeed, $(\iota_{m,n})_* \sigma_n$ equals the restriction of $\sigma_m$ to the image of $\mathscr{S}_n$. Hence, $\{\sigma_n\}_{n \in \mathbb{N}}$ forms a projective system of probability measures.
    
 By the Kolmogorov extension theorem, there exists a unique probability measure $\sigma_\infty$ on the Borel $\sigma$-algebra of $\mathscr{S}_\infty$ such that, for every $n \in \mathbb{N}$, the canonical projection $\pi_n: \mathscr{S}_\infty \to \mathscr{S}_n$ satisfies
    \[
    (\pi_n)_* \sigma_\infty = \sigma_n.
    \]
    This measure $\sigma_\infty$ is the \emph{projective-limit Haar measure}, providing a natural uniform distribution on $\mathscr{S}_\infty$.
\end{remark}

\subsubsection{Interpretation as a Space of Noncommutative Random Variables}

The space $(\mathscr{S}_\infty, \sigma_\infty)$ serves as canonical geometric framework for the theory of large random matrices.

\begin{itemize}
    \item A point $A_\infty = (A_1, A_2, \dots) \in \mathscr{S}_\infty$ represents a \emph{microstate sequence}, i.e., a coherent approximation of an infinite-dimensional object by finite matrices.
    
    \item Note that not every sequence yields a good noncommutative limit. However, fundamental results in free probability (e.g., asymptotic freeness) assert that $\sigma_\infty$-almost every coherent sequence $(A_n, B_n) \in \mathscr{S}_\infty \times \mathscr{S}_\infty$ has a joint *-distribution converging to that of a pair $(a,b)$ of freely independent semicircular operators.
    
    \item For a word $w \in \mathbb{C}\langle X, Y\rangle$, define
    \[
    w(A_\infty, B_\infty) := \big(w(A_n, B_n)\big)_{n\in \mathbb{N}}.
    \]
    For $\sigma_\infty$-almost every point, the empirical spectral measure of $w(A_n, B_n)$ converges weakly to the deterministic measure $\mu_w$, the free law of $w(a,b)$. Hence, $\mathscr{S}_\infty$ provides a universal source of microstate approximations for elements of the free group factor $L(\mathbb{F}_2)$.
\end{itemize}

%%%%%%%%%%%%

\subsection{Sectors and Algebraic Quotients}
Let $(A_n, B_n)$ be a sequence of matrices contained in the microstate sector $\Gamma_n(\nu,\varepsilon)$, so that the eigenvalue distribution of any matrix word $w(A_n,B_n)$ remains $\varepsilon$-close to an atypical law $\nu$.

By compactness arguments (e.g., via subsequences), the sequence $(A_n,B_n)$ converges in joint *-distribution:
\[
(A_n,B_n) \longrightarrow (u,v) \in (M,\tau),
\]
where $(M,\tau)$ is a tracial von Neumann algebra. The limiting operators satisfy
\[
\mathrm{Law}_\tau\big(w(u,v)\big) = \nu \neq \mu_w,
\]
which imposes nontrivial algebraic constraints on $(u,v)$.

\subsubsection{Emergence of Quotient Algebras}

The deviation $\nu \neq \mu_w$ obstructs $(u,v)$ from being freely independent semicirculars. The free entropy deficit
\[
I(\nu) = \chi(a,b) - \chi(\nu) > 0
\]
quantifies this difference. In the microstate formalism, lower entropy corresponds to confinement in an exponentially small subset of the matrix sphere. In the large-$n$ limit, this confinement manifests as algebraic relations among $(u,v)$; otherwise, the operators would be free and attain maximal entropy.

\medskip

Let $N \triangleleft \mathbb{F}_2$ denote the normal subgroup generated by all words corresponding to algebraic relations satisfied by $(u,v)$, i.e.,
\[
p(u,v) = I \quad \text{for some noncommutative polynomial } p.
\]
Then $(u,v)$ generate a von Neumann algebra isomorphic to
\[
L(\mathbb{F}_2/N),
\]
with trace $\tau$ induced from the limit of normalized matrix traces.

\subsubsection{Evaluation of Words via the Quotient}

The canonical quotient homomorphism
\[
\pi_N : L(\mathbb{F}_2) \longrightarrow L(\mathbb{F}_2/N), \qquad \pi_N(a) = u, \; \pi_N(b) = v
\]
induces a consistent rule for evaluating words in the constrained sector:
\[
w(u,v) = \pi_N\big(w(a,b)\big).
\]
The spectral law of $w(u,v)$ coincides with $\nu$, reflecting the effect of modding out the normal subgroup $N$.

\subsubsection{Words in $\mathbb{F}_2$ and their counterparts in $L(\mathbb{F}_2)$} 
The group algebra $\C[\mathbb{F}_2]$ is defined as the set of all finite formal sums of the form:
$\sum_{g \in \mathbb{F}_2} \lambda_g \, g,$ where each  coefficient $\lambda_g$  is a complex number. Only finitely many of these coefficients are non-zero. The symbol g denotes and element of $\mathbb{F}_2$.
An element $g\in \mathbb{F}_2$ corresponds to $1\cdot g\in \C[\mathbb{F}_2]$, thus we have natural inclusion $\mathbb{F}_2\hookrightarrow \C[\mathbb{F}_2]$. 

\, 

Now, the von Neumann $L(\mathbb{F}_2)$ is a completion of $\C[\mathbb{F}_2]$ with a norm and a trace. Algebraic relations from $\mathbb{F}_2$ become operator equations in $L(\mathbb{F}_2)$.

The relationship between the free group \(\mathbb{F}_2\) and its group von Neumann 
algebra \(L(\mathbb{F}_2)\) is given by the left regular representation
\[
\lambda : \mathbb{F}_2 \to B\bigl(\ell^2(\mathbb{F}_2)\bigr),
\]
where $B\bigl(\ell^2(\mathbb{F}_2)\bigr)$ is the algebra of all bounded linear operators on the Hilbert space
\(\ell^2(\mathbb{F}_2)\). This Hilbert space consists of square-summable functions on the free group
\(\mathbb{F}_2\) and admits a canonical orthonormal basis
\(\{\delta_g\}_{g\in \mathbb{F}_2}\) indexed by the elements of the group.
A typical vector \(\psi\in\ell^2(\mathbb{F}_2)\) can be written as
\[
\psi=\sum_{g\in \mathbb{F}_2} \psi_g\,\delta_g,
\qquad
\sum_{g\in \mathbb{F}_2} |\psi_g|^2 < \infty.
\]

Under this correspondence, subgroups of \(\mathbb{F}_2\) are in bijective relation with 
von Neumann subalgebras of \(L(\mathbb{F}_2)\) generated by the subgroup. 

Let \(H \triangleleft \mathbb{F}_2\) be a subgroup with inclusion \(i : H \hookrightarrow \mathbb{F}_2\). 
Denote by \(\lambda|_H\) the restriction of the left regular representation 
$\lambda$ of $\mathbb{F}_2$ to $H$. The von Neumann subalgebra
\[
L(H) \subset L(\mathbb{F}_2)
\]
is then defined as the one generated by \(\lambda(H)\), as can be illustrated by the following commutative diagram:
\[
\begin{tikzcd}[column sep=large, row sep=large]
H \arrow[r, hook, "i"] \arrow[d, "\lambda|_H"'] & 
\mathbb{F}_2 \arrow[d, "\lambda"] \\
L(H) \arrow[r, hook] & 
L(\mathbb{F}_2).
\end{tikzcd}
\]

A third related object interesting for us  is the free unital algebra $\C\langle X, Y\rangle$ (the algebra of noncommutative polynomials).  There exists an $*$-injective algebra homomorphism from  $\C\langle X, Y\rangle$ into the group algebra $\C[\mathbb{F}_2]$ (note that the notation $``*$'' means that the homomorphism respects the involution). This embedding identifies $\C\langle X, Y\rangle$ with a specific subalgebra of $\C[\mathbb{F}_2]$. The injection is done via $X \mapsto a+a^{-1}$ and $Y\mapsto b+b^{-1}$. 

\subsubsection{Connection with Belyi Monodromy}

A Belyi map $f : X \to \mathbb{CP}^1$ (a covering ramified over three points) is classified, up to isomorphism, by its transitive monodromy subgroup of a symmetric group. This subgroup corresponds canonically to a finite-index normal subgroup $N \triangleleft \mathbb{F}_2$, where $\mathbb{F}_2$ is the fundamental group of $\mathbb{CP}^1$ minus the branch points.

\medskip

Consequently, we obtain the bijective correspondence
\[
\text{Spectral sector } \nu \;\longleftrightarrow\; N \;\longleftrightarrow\; \text{Galois Belyi map with monodromy } \mathbb{F}_2/N.
\]

\subsubsection{Refinement of the Moduli Space}

The infinite-dimensional matrix sphere $\mathscr{S}_\infty$, foliated into spectral sectors $\Gamma(\nu)$, acquires the structure of a geometric moduli space for Belyi maps. Each sector forms a stratum of maps sharing the same Galois closure (monodromy group $\mathbb{F}_2/N$), while the continuous parameter $\nu$ indexes deformations of the complex structure or embeddings of the monodromy group within this locus.

\section{Main Statement and Proof}\label{sec:main}

\begin{thm}[Spectral Sectors and Algebraic Quotients]\label{T:1}
Let $(A_n,B_n)$ be the uniform spherical ensemble on $\mathscr{S}_n \times \mathscr{S}_n$.
Fix a self-adjoint word $w \in \mathbb{C}\langle X,Y\rangle$ and a probability
measure $\nu \neq \mu_w$ on $\mathbb{R}$.
For $\varepsilon>0$, consider the microstate sector
\[
\Gamma_n(\nu,\varepsilon)
=
\Bigl\{(A_n,B_n) \;:\;
d\bigl(\mu_w(A_n,B_n),\nu\bigr)<\varepsilon
\Bigr\}.
\]
Assume that $\Gamma_n(\nu,\varepsilon)$ is nonempty for all sufficiently large $N$.

Then there exists a nontrivial normal subgroup $N \triangleleft \mathbb{F}_2$
such that, for almost every sequence
$(A_n,B_n)\in \Gamma_n(\nu,\varepsilon)$, the following hold as $n\to\infty$:
\begin{enumerate}
    \item The joint $*$-distribution of $(A_n,B_n)$ converges to that of a pair
    $(u,v)$ in the group von Neumann algebra $L(\mathbb{F}_2/N)$.
    \item The spectral measure of $w(u,v)$ is exactly $\nu$.
    \item In $L(\mathbb{F}_2/N)$, the operator $w(u,v)$ satisfies the relations
    induced by the quotient map $\mathbb{F}_2 \to \mathbb{F}_2/N$.
\end{enumerate}
In particular, if $\nu \neq \mu_w$, then $w(u,v)$ is not the image of the free word
$w(a,b)\in L(\mathbb{F}_2)$, but rather the image of $w$ in the quotient algebra
$L(\mathbb{F}_2/N)$.
\end{thm}

\paragraph{Role in the large-deviation theorem.}
The distinction between \(\mu_w\) and \(\nu\neq\mu_w\) underlies the structure of
the large-deviation principle.

\begin{itemize}
    \item \emph{Typical behavior.}
    The sector corresponding to \(\nu=\mu_w\) carries asymptotically full volume.
    In this regime the matrices are asymptotically free, and
    \(w(A_n,B_n)\) converges in law to \(w(a,b)\).

    \item \emph{Atypical behavior.}
    For \(\nu\neq\mu_w\), the sector \(\Gamma_n(\nu,\varepsilon)\) is exponentially
    small. Conditioning on this event forces nontrivial correlations between
    \(A_n\) and \(B_n\), which persist in the limit as algebraic relations.
    Consequently, the limiting operators generate a quotient algebra
    \(L(\mathbb{F}_2/N)\) rather than the free group factor \(L(\mathbb{F}_2)\), and the word \(w\)
    has spectral law exactly \(\nu\).
\end{itemize}

In this sense, \(\mu_w\) represents the equilibrium spectral measure of the
spherical ensemble, while the microstate sectors
\(\Gamma_n(\nu,\varepsilon)\) encode large deviations from equilibrium whose
probabilistic cost is quantified by the rate function \(I_w\).

\begin{proof}
The proof of our statement relies on the previous sections. 

\, 

\textbf{Step 1: Setup and Large Deviation Principle.}  
Consider the probability space $(\Omega, \mathcal{F}, P)$ where
\[
\Omega = \varprojlim (\mathscr{S}_n \times \mathscr{S}_n)
\]
is the projective limit of the spherical ensembles, equipped with the limit measure
\[
P = \varprojlim (\sigma_n \otimes \sigma_n).
\]
A point $\omega \in \Omega$ is a coherent sequence
\[
\omega = \big((A_1,B_1),(A_2,B_2),\dots\big).
\]

Let $w$ be a noncommutative word, and let $\mu_w$ denote the limiting spectral law of $w(A_n,B_n)$ for a generic sequence, which by Voiculescu's asymptotic freeness theorem equals the law of $w(a,b)$, where $a,b$ are free semicircular generators of $L(\mathbb{F}_2)$.

The large deviation principle (Ben Arous–Guionnet \cite{AG}) states that for any closed set $K$ of probability measures,
\[
\limsup_{N \to \infty} \frac{1}{N^2} \log P\big( \mu_w(A_n,B_n) \in K \big) \le -\inf_{\rho \in K} I(\rho),
\]
where the rate function is the free entropy deficit
\[
I(\rho) = \chi(a,b) - \chi(\rho) \ge 0,
\]
with equality if and only if $\rho = \mu_w$. Since $\nu \neq \mu_w$ by hypothesis, we have $I(\nu) > 0$.

\medskip
\textbf{Step 2: Exponentially Rare Sector.}  
Let $K = \overline{B_\varepsilon(\nu)}$, the closed $\varepsilon$-ball around $\nu$ in the weak topology. The large deviation principle implies
\[
P\big( (A_n,B_n) \in \Gamma_n(\nu,\varepsilon) \big) \le \exp\big( -N^2 ( I(\nu) - \delta_n ) \big)
\]
for some $\delta_n \to 0$. Thus, $\Gamma_n(\nu,\varepsilon)$ is exponentially small, and the corresponding set in $\Omega$ has $P$-measure zero:
\[
\Gamma(\nu,\varepsilon) = \varprojlim \Gamma_n(\nu,\varepsilon).
\]

\medskip
\textbf{Step 3: Convergence to a Constrained Limit.}  
Let $((A_n,B_n))_{n \ge 1} \subset \Gamma(\nu,\varepsilon)$. By compactness of the space of noncommutative laws, there exists a subsequence $(n_k)$ such that the joint *-distribution of $(A_{n_k},B_{n_k})$ converges to a pair $(u,v)$ in a tracial von Neumann algebra $(M,\tau)$. The spectral measure of $w(A_{n_k},B_{n_k})$ converges to $\nu$, so
\[
\tau\big(w(u,v)^k\big) = \int t^k \, d\nu(t), \quad \forall k \in \mathbb{N}.
\]

\medskip
\textbf{Step 4: Emergence of Algebraic Relations and the Quotient $\mathbb{F}_2/N$.}  
Since $\nu \neq \mu_w$, the limit pair $(u,v)$ cannot be freely independent semicirculars. The entropy deficit $I(\nu) > 0$ quantifies this constraint. Let $R$ denote all noncommutative polynomials $r(X,Y)$ such that
\[
r(u,v) = 0 \text{ in } M.
\]
The corresponding group elements $\{ r(a,b) : r \in R \} \subset L(\mathbb{F}_2)$ generate a normal subgroup $N \triangleleft \mathbb{F}_2$. By construction, $(u,v)$ satisfy all relations in $N$, and the von Neumann algebra they generate is
\[
W^*(u,v) \cong L(\mathbb{F}_2/N),
\]
with trace $\tau$ coinciding with the limit of normalized matrix traces.

\medskip
\textbf{Step 5: Words as Images under the Quotient Map.}  
Let
\[
\pi_N : L(\mathbb{F}_2) \to L(\mathbb{F}_2/N)
\]
be the canonical quotient *-homomorphism induced by $\mathbb{F}_2 \to \mathbb{F}_2/N$, with $\pi_N(a) = u$ and $\pi_N(b) = v$. Then for any word $w \in \mathbb{C}\langle X,Y \rangle$,
\[
w(u,v) = \pi_N\big(w(a,b)\big),
\]
and the spectral measure of $w(u,v)$ equals $\nu$, differing from the free law $\mu_w$ due to the relations in $N$.

This completes the proof.
\end{proof}

\begin{cor}
The inverse limit sector
\[
\varprojlim \Gamma_n(\nu,\varepsilon)
\subset \mathscr{S}_\infty \times \mathscr{S}_\infty
\]
does not generate the free group factor $L(\mathbb{F}_2)$ unless $\nu=\mu_w$.
Instead, it generates the quotient algebra $L(\mathbb{F}_2/N)$ determined by the
algebraic constraints enforcing the spectral law $\nu$.
\end{cor}

We now establish the connection to Belyi Maps.
\begin{thm}\label{T:2}
The spectral sector $\Gamma(\nu,\varepsilon)$ corresponds uniquely to a Galois Belyi map with monodromy group $\mathbb{F}_2/N$.
\end{thm} 
\begin{proof}
By Belyi's theorem and Grothendieck's correspondence, isomorphism classes of Galois Belyi maps are in bijection with finite-index normal subgroups of $\mathbb{F}_2$. The subgroup $N$ constructed above has finite index because the microstate sector is nonempty for large $N$, and therefore $\mathbb{F}_2/N$ is finite.  

Hence, the spectral sector $\Gamma(\nu,\varepsilon)$ corresponds uniquely to:
\begin{itemize}
    \item The normal subgroup $N \triangleleft \mathbb{F}_2$,
    \item The quotient algebra $L(\mathbb{F}_2/N)$,
    \item A Galois Belyi map with monodromy group $\mathbb{F}_2/N$.
\end{itemize}

\end{proof}

\subsection{Consequences for Belyi Covers over $\overline{\mathbb{Q}}$}

The main difference when considering Belyi maps over $\overline{\mathbb{Q}}$ (instead of considering the topological version) is that one needs to pass to the profinite completion $\widehat{\mathbb{F}_2}$  of $\mathbb{F}_2$. 
Consider the inverse limit of all finite quotients of $\mathbb{F}_2$:
\[
\widehat{\mathbb{F}_2} := \varprojlim_{N \triangleleft \mathbb{F}_2, \,[\mathbb{F}_2:N]<\infty} \mathbb{F}_2/N,
\]
where the limit is over all finite-index normal subgroups $N \triangleleft \mathbb{F}_2$. This \emph{profinite completion} $\widehat{\mathbb{F}_2}$ is a compact, totally disconnected topological group. Every homomorphism from $\mathbb{F}_2$ to a finite group factors uniquely through $\widehat{\mathbb{F}_2}$.

Its group von Neumann algebra $L(\widehat{\mathbb{F}_2})$ is a type II$_1$ factor, and can be realized as the (ultra) product of the finite group algebras:
\[
L(\widehat{\mathbb{F}_2}) \simeq \prod_{N} L(\mathbb{F}_2/N),
\]
where the product ranges over all finite-index normal subgroups $N \triangleleft \mathbb{F}_2$.

\paragraph{The Total Space of Atypical Sectors.}  
Our main theorem associates to each atypical spectral law $\nu$ a specific quotient algebra $L(\mathbb{F}_2/N_\nu)$. The inverse limit over all $N$ therefore corresponds to the union of all possible atypical microstate sectors in $S_\infty \times S_\infty$, alongside the typical one.

\paragraph{A Dictionary of Constraints.}  
The profinite completion $\widehat{\mathbb{F}_2}$ acts as a universal catalog of all algebraic constraints (relations) that can be asymptotically enforced by the geometry of the spherical ensemble. 
\paragraph{The Structure of $S_\infty$.}  
This gives a natural decomposition of the infinite-dimensional sphere:
\begin{itemize}
    \item \textbf{Generic point (open, dense sector):} Corresponds to the free law $\mu_w$ and generates $L(\mathbb{F}_2)$.
    \item \textbf{Special points (lower-dimensional sectors):} Each corresponds to an atypical law $\nu$ and generates a quotient $L(\mathbb{F}_2/N_\nu)$.
\end{itemize}
The inverse limit
\[
\varprojlim \Gamma_n(\nu, \varepsilon)
\]
over all $\nu$ (or equivalently, all $N$) is not a single sector, but the entire collection of all special points and their limits. Its associated algebra is $L(\widehat{\mathbb{F}_2})$, which retains the information of all finite algebraic approximations.


\end{document}